\documentclass[11pt,a4paper]{article}
\usepackage[T1]{fontenc}
\usepackage[utf8]{inputenc}
\usepackage{lmodern}
\usepackage{mathtools,amssymb,amsthm}
\usepackage{microtype}
\usepackage[shortlabels]{enumitem}
\usepackage[a4paper,margin=25mm]{geometry}
\usepackage{setspace}
\usepackage{placeins}
\usepackage{tikz}
\usetikzlibrary{fit,backgrounds,calc}
\usepackage[hidelinks]{hyperref}
\hypersetup{
 pdftitle={A linear upper bound for Berge saturation numbers},
 pdfauthor={Tianying Xie},
 pdfsubject={Saturation numbers for Berge hypergraphs},
 pdfkeywords={Berge hypergraphs, saturation, vertex cover, matching}
}
\allowdisplaybreaks[1]
\numberwithin{equation}{section}
\newtheorem{theorem}{Theorem}[section]

\newtheorem{lemma}[theorem]{Lemma}
\newtheorem{conjecture}[theorem]{Conjecture}
\theoremstyle{definition}
\newtheorem{construction}[theorem]{Construction}
\theoremstyle{plain}
\DeclareMathOperator{\sat}{sat}
\newcommand{\BF}{\textup{Berge-}F}
\newcommand{\HH}{\mathcal H}
\newcommand{\JJ}{\mathcal J}
\newcommand{\disj}{\mathbin{\dot\cup}}

\title{A linear upper bound for Berge saturation numbers}
\author{Tianying Xie\footnote{\raggedright School of Mathematics and Statistics,
Fuzhou University, Fuzhou, Fujian 350108, China
(\texttt{xiety@fzu.edu.cn}). Research supported by the National Key
Research and Development Program of China under grant No.~2023YFA1010201
and the National Natural Science Foundation of China under grant
No.~12501474.}}
\date{}

\begin{document}
\maketitle
\begin{abstract}
For a graph $F$ with at least one edge, a $k$-uniform hypergraph is
Berge-$F$-saturated if it contains no Berge-$F$, but adding any missing
hyperedge creates a Berge-$F$. The saturation number $\sat_k(n,\BF)$
is the minimum number of edges in such a hypergraph on $n$ vertices.
English, Gordon, Graber, Methuku and Sullivan conjectured a linear upper
bound for every fixed finite family of forbidden graphs. We prove the
single-graph case: $\sat_k(n,\BF)=O_{F,k}(n)$ for every fixed integer
$k\ge2$ and every fixed finite simple graph $F$ with at least one edge. 
The proof combines a sparse construction with degree and matching arguments.
\end{abstract}

\section{Introduction}\label{sec:introduction}

Let $F$ be a finite simple graph and let $\HH$ be a hypergraph with
$V(F)\subseteq V(\HH)$. We say that
$\HH$ is a \emph{Berge-$F$} if there is a bijection
$\psi:E(F)\longrightarrow E(\HH)$ such that $e\subseteq\psi(e)$
for every $e\in E(F)$. We call $F$ a \emph{core graph} of $\HH$
and its vertices the \emph{core vertices}.
For $k\ge2$, let $\mathcal B_k(F)$ denote the family of $k$-uniform
Berge-$F$ hypergraphs, considered up to isomorphism, with isolated
vertices outside the core graph omitted.
A $k$-uniform hypergraph \emph{contains a Berge-$F$} if it has a
subhypergraph isomorphic to a member of $\mathcal B_k(F)$. 
The hypergraph is \emph{Berge-$F$-free} if it contains no Berge-$F$.
It is \emph{Berge-$F$-saturated} if
it is Berge-$F$-free and adding any missing $k$-edge creates a Berge-$F$.
For a graph $F$, the saturation number is
\[
 \sat_k(n,\BF)=\min\{|E(\HH)|:|V(\HH)|=n,\ \HH\text{ is $k$-uniform and
 Berge-$F$-saturated}\}.
\]

For $k=2$, this is ordinary graph saturation, and we write
$\sat(n,F)=\sat_2(n,\BF)$. Its study was initiated by
Erd\H{o}s, Hajnal and Moon~\cite{EHM}. K\'aszonyi and
Tuza~\cite{KT} proved that $\sat(n,F)=O_F(n)$ for every fixed graph
$F$ with at least one edge. For $k$-uniform hypergraphs, Pikhurko~\cite{Pik} proved an
$O(n^{k-1})$ upper bound for every fixed finite family of forbidden
hypergraphs. Every member of $\mathcal B_k(F)$ has at most
$|V(F)|+(k-2)|E(F)|$ vertices, so $\mathcal B_k(F)$ is finite.
Pikhurko's result therefore gives
$\sat_k(n,\BF)=O_{F,k}(n^{k-1})$.

English, Gordon, Graber, Methuku and Sullivan~\cite{EGGMS}
initiated the study of Berge saturation and obtained results for
paths, cycles, stars, triangles and matchings. They proposed the
following conjecture. For a family $\mathcal F$ of graphs, put
$\mathcal B_k(\mathcal F)=\bigcup_{F\in\mathcal F}\mathcal B_k(F)$.
A hypergraph is Berge-$\mathcal F$-free if it is Berge-$F$-free for
every $F\in\mathcal F$, and is Berge-$\mathcal F$-saturated if it is
Berge-$\mathcal F$-free and adding any missing $k$-edge creates a
Berge-$F$ for some $F\in\mathcal F$. We write
$\sat_k(n,\textup{Berge-}\mathcal F)=\sat_k(n,\mathcal B_k(\mathcal F))$.

\begin{conjecture}[English et al.~{\cite[Conjecture~6.3]{EGGMS}}]\label{conj:family}
For every fixed integer $k\ge2$ and every fixed finite nonempty family
$\mathcal F$ of graphs, each having at least one edge,
\[
 \sat_k(n,\textup{Berge-}\mathcal F)=O_{\mathcal F,k}(n).
\]
\end{conjecture}

English, Gerbner, Methuku and Tait~\cite{EGMT} proved the single-graph
case for all $3\le k\le5$. For arbitrary $k\ge3$, they established
the linear bound when $F$ has a tree component, when its vertex cover
number $\beta(F)$ is at least $k+1$, or when it has a component with
vertex feedback number at most two. Gerbner, Patk\'os, Tuza and
Vizer~\cite[Theorems~1.2 and~1.3]{GPTV} subsequently proved linear bounds
for complete multipartite graphs with at least three nonempty parts,
and for graphs having an edge whose endpoint degrees are the two
smallest entries of the degree sequence.
English, Kritschgau, Nahvi and Sprangel~\cite{EKNS} obtained further
linear bounds when $\delta(F)>\beta(F)/2$ or $g(F)>f(F)$, where
$\delta(F)$ is the minimum degree, $g(F)$ is the girth and $f(F)$
is the minimum number of vertices whose deletion leaves a forest.
Exact and asymptotic saturation numbers, as well as saturation spectra,
have also been studied for Berge stars and cliques;
see~\cite{AE,BEHJR,EKNS,CHX}.

In this paper, we prove the single-graph case of
Conjecture~\ref{conj:family} for every uniformity.

\begin{theorem}\label{thm:main}
For every fixed integer $k\ge2$ and every fixed finite simple graph $F$
with at least one edge,
\[
 \sat_k(n,\BF)=O_{F,k}(n).
\]
\end{theorem}

Our proof builds on the construction of English, Gerbner, Methuku and
Tait~\cite[Construction~2.1]{EGMT} and their argument for bounding the
number of additional edges. The case in which we delete two adjacent
vertices is also related to the type-I construction of Gerbner,
Patk\'os, Tuza and Vizer~\cite[proof of Theorem~1.3]{GPTV}.
We replace the full family of subsets of a prescribed size by a selected
family that admits the required matchings.

The new step is to make this choice uniformly in the remaining range
$\beta(F)\le k$. A degree threshold identifies either two adjacent
vertices or an induced star whose vertices can be deleted from $F$.
In the star case, we choose subsets of size strictly less than $k$.
When more than $k$ high-degree vertices remain after deleting the star,
a minimum vertex cover gives initial indexed subsets of size at most
$\beta(F)-1$. We enlarge them to $(k-1)$-subsets and make them distinct
by an exchange argument, preserving the required matchings.
The strict bound below $k$ leaves room for a nonempty part of outside
vertices in each hyperedge. A count of high-degree vertices then excludes
Berge-$F$ from the construction, and a greedy forest embedding bounds
all Berge-$F$-free extensions.

Section~\ref{sec:prelim} records the known cases and proves the exchange
lemma. Section~\ref{sec:proof} gives the construction, verifies it, and
chooses its parameters to prove Theorem~\ref{thm:main}.

\section{Preliminaries}\label{sec:prelim}

All graphs and hypergraphs are finite and simple. A hypergraph is
$k$-uniform if each hyperedge has size $k$.
For a graph $F$, let $V(F)$ and $E(F)$ denote its vertex and edge sets,
and let $N_F(v)$ and $d_F(v)$ denote the neighborhood and degree of $v$.
The degree $d_{\HH}(v)$ of a vertex in a hypergraph is the number of
hyperedges containing it. For $W\subseteq V(\HH)$, the induced
subhypergraph $\HH[W]$ has vertex set $W$ and edge set
$\{h\in E(\HH):h\subseteq W\}$. For $R\subseteq V(F)$, let $F[R]$ denote
the induced subgraph, and write $F-R=F[V(F)\setminus R]$.
For disjoint vertex sets $U,W\subseteq V(F)$, let $e_F(U,W)$ be the
number of edges with one endpoint in each set.
A \emph{vertex cover} is a set of vertices meeting every edge;
its minimum size is the \emph{vertex cover number}, denoted by $\beta(F)$.
A forest is an acyclic graph and may have isolated vertices.
For a nonnegative integer $d$, a hypergraph $\HH$ is
\emph{$d$-degenerate} if every nonempty induced subhypergraph has a
vertex of degree at most $d$. 
For a positive integer $q$, put $[q]=\{1,\ldots,q\}$, and let
$\binom St$ denote the family of $t$-subsets of $S$.
For nonnegative integers $b,t$ with $t>b$, we set $\binom bt=0$.

We may assume that $F$ has no isolated
vertices, since they do not affect Berge saturation for sufficiently
large $n$. Vertices that become isolated in $F-R$ are nevertheless
retained in the construction. All constants may depend on $F$ and $k$,
but not on $n$.

We use the known cases without reproving them. The case $k=2$ follows
from~\cite{KT}. For $k\ge3$, graphs having a tree component are covered
by~\cite[Observation~2.4]{EGMT}, and the range $\beta(F)\ge k+1$
is covered by~\cite[Theorem~2.5]{EGMT}. Thus it remains to consider
the following range:
\begin{equation}\label{eq:remaining}
 k\ge3,\qquad \beta(F)\le k,\qquad
 \text{every component of $F$ contains a cycle}.
\end{equation}

\subsection{Choosing distinct subsets}

For an indexed family $P_1,\ldots,P_r\subseteq S$, its incidence graph
is the bipartite graph with vertex classes $S$ and $[r]$, where $z$
is adjacent to $i$ if and only if $z\in P_i$. Equal sets $P_i=P_j$
still correspond to distinct vertices $i$ and $j$. The following
lemma allows us to make the sets distinct while preserving the
matchings needed in the construction.

\begin{lemma}\label{lem:distinct}
Let $S$ be a set of size $c$, let $t$ be a nonnegative integer, and
let $P_1,\ldots,P_r\subseteq S$ be indexed sets satisfying
\[
 |P_i|\le t\le c,\qquad 1\le r\le\binom ct.
\]
There are pairwise distinct sets $P'_1,\ldots,P'_r\in\binom St$ with
the following property: for every $Q\subseteq S$, if the incidence graph
of $(P_1,\ldots,P_r)$ has a matching saturating $Q$, then so does the
incidence graph of $(P'_1,\ldots,P'_r)$.
\end{lemma}

\begin{proof}
First enlarge each $P_i$ to a $t$-set. This only adds edges to the
incidence graph, so every previously matchable subset of $S$ remains
matchable. If $t=0$ or $t=c$, the inequality $r\le\binom ct$ gives
$r=1$, and the enlarged family already has the required property.
We may therefore assume $0<t<c$.

We first verify that one elementary replacement preserves all the
required matchings. Suppose that two indices $i\ne j$ have the same
set $P_i=P_j=P$. Choose $x\in P$ and $y\in S\setminus P$, and replace
only $P_i$ by
\(
 P_i^*=(P\setminus\{x\})\cup\{y\}.
\)
Let $Q\subseteq S$ be any set that has a matching in the incidence
graph before this replacement, and fix a matching saturating $Q$.
The only incidence removed is the edge joining $x$ to index $i$.
If this edge is not in the matching, the matching remains valid.
Otherwise, match $x$ to index $j$, which still corresponds to $P$.
If $j$ was unmatched, this completes the modification. If $j$ was
matched to a vertex $z$, then $z\ne x$ and $z\in P$; match $z$ to
index $i$ instead, since $z\in P_i^*$. The modified matching still
saturates $Q$. Thus this replacement preserves the existence of a
matching for every subset $Q$ that was matchable before it.

To eliminate all repetitions using these replacements, consider the
Johnson graph $J(c,t)$: its vertices are the $t$-subsets of $S$, and
two vertices are adjacent if one is obtained from the other by
replacing a single element. This graph is connected. Indeed, for
two $t$-sets $U$ and $V$, we can successively replace an element of
$U\setminus V$ by an element of $V\setminus U$ until we reach $V$.

Place each index $i\in[r]$ at the vertex $P_i$ of $J(c,t)$. If a vertex
carries at least two indices, some vertex is unoccupied, because
$r\le\binom ct$. Fix a vertex carrying at least two indices, and choose
an unoccupied vertex at minimum distance from it. Every internal
vertex of a shortest path between them is occupied. Move one index
along this path, one edge at a time. Before each move its current
vertex carries at least two indices, so the replacement above applies.
At the end, the number of occupied vertices has increased by one,
and all previously matchable subsets remain matchable. Repeat until
the $r$ indices occupy distinct vertices.
\end{proof}

\section{The construction and proof of Theorem~\ref{thm:main}}
\label{sec:proof}

We first give the construction and prove that every Berge-$F$-free
hypergraph containing it has only linearly many edges. The bound on
the additional edges follows the degeneracy argument
in~\cite[Lemma~2.2]{EGMT}; we include the argument for the more general
choice of subsets used here. We then choose the parameters in two
cases to prove Theorem~\ref{thm:main}.

\subsection{The construction}\label{sec:construction}

\begin{construction}\label{con:initial}
Fix an integer $d$, sets $A,R,S\subseteq V(F)$, a graph $G$,
integers $q,t$, and subsets $P_1,\ldots,P_q$ satisfying the following
conditions:
\begin{enumerate}[label=\arabic*.,leftmargin=*,itemsep=.4\baselineskip]
\item $d\ge2$,
\[
 A=\{v\in V(F):d_F(v)\ge d\},\qquad R\cap A\ne\varnothing,
 \qquad S=A\setminus R,\qquad G=F-R.
\]
\item The graph $F[R]$ is a forest, and every edge between $R$ and
$V(G)$ has its endpoint in $V(G)$ in $S$.
\item The integers $q,t$ satisfy
\begin{equation}\label{eq:parameters} 
  1\le q\le d-1,\qquad 0\le t<k.
\end{equation}
\item The sets $P_1,\ldots,P_q$ are pairwise distinct $t$-subsets of $S$.
For each $v\in R$, put $Q_v=N_F(v)\cap S=N_F(v)\setminus R$.
We require an injection
\begin{equation}\label{eq:matching}
 \rho_v:Q_v\longrightarrow[q],\qquad
 z\in P_{\rho_v(z)}\quad(z\in Q_v).
\end{equation}
Equivalently, the incidence graph of the $P_j$ has a matching
saturating each $Q_v$. These matchings are chosen separately for
different vertices $v$.
\end{enumerate}
All the parameters above are fixed independently of $n$.
Put $a=k-t\ge1$. With these choices, for sufficiently large $n$,
construct a $k$-uniform hypergraph $\HH_0$ on an $n$-vertex set
as follows. Choose $|V(G)|$ vertices and label them by $V(G)$,
retaining any isolated vertices of $G$. For each $xy\in E(G)$,
choose a set $W_{xy}$ of $k-2$
additional vertices. These sets are pairwise disjoint and disjoint
from $V(G)$. Put
\[
 h_{xy}=\{x,y\}\cup W_{xy},\qquad
 Z_0=V(G)\cup\bigcup_{xy\in E(G)}W_{xy}.
\]
The hyperedges $h_{xy}$ form the $k$-uniform expansion of $G$.
Partition the remaining vertices into sets $B_1,\ldots,B_m$
of size $a$ and a remainder $Z_1$ of size less than $a$, where
\[
 m=\left\lfloor\frac{n-|Z_0|}{a}\right\rfloor,
 \qquad |Z_0|=|V(G)|+(k-2)|E(G)|.
\]
Thus
\[
 V(\HH_0)
 =Z_0\disj B_1\disj\cdots\disj B_m\disj Z_1,
 \qquad |B_i|=a,\quad |Z_1|<a.
\]
Set $Z=Z_0\disj Z_1$ and define the edge set by
\begin{equation}\label{eq:construction}
 E(\HH_0)
 =\{h_{xy}:xy\in E(G)\}
 \ \cup\ \{B_i\cup P_j:i\in[m],\ j\in[q]\}.
\end{equation}
\end{construction}

Every hyperedge has size $k$, and all hyperedges are distinct: the
$P_j$ are distinct, every $B_i$ is nonempty, and no $h_{xy}$ meets a
part $B_i$. The vertices of $Z_1$ are isolated. Since $G$, $q$ and $a$
are fixed independently of $n$, we have
\begin{equation}\label{eq:initial-size}
 |E(\HH_0)|=|E(G)|+qm
 =\frac qa n+O_{F,k}(1)=O_{F,k}(n).
\end{equation}
Moreover,
\begin{equation}\label{eq:fixed-size}
 |Z|\le |V(G)|+(k-2)|E(G)|+a-1.
\end{equation}
Thus $|Z|$ is bounded independently of $n$, and the remainder allows
all sufficiently large values of $n$.
An example is shown in Figure~\ref{fig:example}.

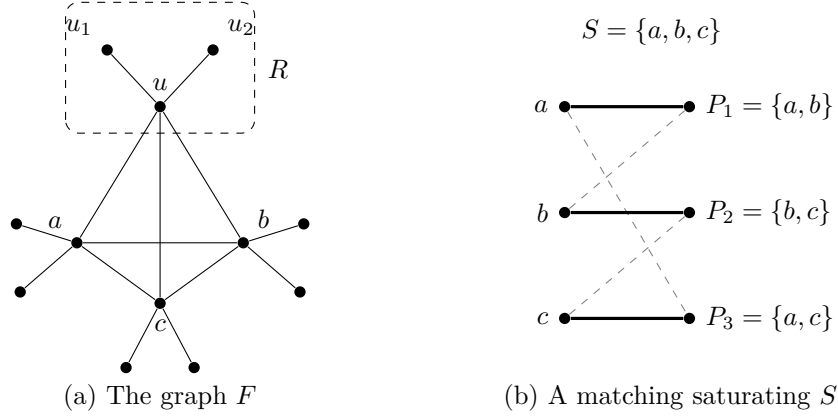
\begin{figure}[!htbp]
\centering
\begin{tikzpicture}[x=1cm,y=1cm,
 vertex/.style={circle,fill=black,inner sep=1.5pt},
 every node/.style={font=\small}]
 \begin{scope}
 \node[vertex,label=above:$u$] (u) at (0,1.8) {};
 \node[vertex,label=above left:$a$] (a) at (-1.1,0) {};
 \node[vertex,label=above right:$b$] (b) at (1.1,0) {};
 \node[vertex,label=below:$c$] (c) at (0,-0.8) {};
 \foreach \x/\y in {u/a,u/b,u/c,a/b,a/c,b/c}
   \draw (\x)--(\y);
 \node[vertex,label=above left:$u_1$] (u1) at (-0.7,2.55) {};
 \node[vertex,label=above right:$u_2$] (u2) at (0.7,2.55) {};
 \draw (u1)--(u)--(u2);
 \foreach \name/\xx/\yy/\parent in
  {a1/-1.9/0.25/a,a2/-1.85/-0.65/a,b1/1.9/0.25/b,b2/1.85/-0.65/b,c1/-0.45/-1.65/c,c2/0.45/-1.65/c}{
  \node[vertex] (\name) at (\xx,\yy) {};
  \draw (\parent)--(\name);
 }
 \draw[dashed,rounded corners=5pt] (-1.25,1.45) rectangle (1.25,3.18);
 \node[anchor=west] at (1.29,2.3) {$R$};
 \node at (0,-2.05) {(a) The graph $F$};
 \end{scope}
 \begin{scope}[shift={(6.0,0)}]
 \node at (0.5,2.8) {$S=\{a,b,c\}$};
 \node[vertex,label=left:$a$] (sa) at (-0.65,1.8) {};
 \node[vertex,label=left:$b$] (sb) at (-0.65,0.4) {};
 \node[vertex,label=left:$c$] (sc) at (-0.65,-1.0) {};
 \node[vertex,label={right:$P_1=\{a,b\}$}] (p1) at (1.0,1.8) {};
 \node[vertex,label={right:$P_2=\{b,c\}$}] (p2) at (1.0,0.4) {};
 \node[vertex,label={right:$P_3=\{a,c\}$}] (p3) at (1.0,-1.0) {};
 \draw[gray,dashed] (sb)--(p1) (sc)--(p2) (sa)--(p3);
 \draw[very thick] (sa)--(p1) (sb)--(p2) (sc)--(p3);
 \node at (0.75,-2.05) {(b) A matching saturating $S$};
 \end{scope}
\end{tikzpicture}
\caption{Let $F$ be obtained from $K_4$ by adding two pendant vertices at
each vertex, and let $k=4$. Take $d=5$, $R=\{u,u_1,u_2\}$ and
$S=\{a,b,c\}$. The displayed sets satisfy~\eqref{eq:matching}:
$Q_u=S$ and $Q_{u_1}=Q_{u_2}=\varnothing$.
In the resulting hypergraph $\HH_0$, each part $B_i$ has size two
and contributes the three edges
$B_i\cup\{a,b\}$, $B_i\cup\{b,c\}$ and $B_i\cup\{a,c\}$.
The hyperedges $h_{xy}$ form the $k$-uniform expansion of $F-R$.
In (b), lines are edges of the incidence graph, not hyperedges.}
\label{fig:example}
\end{figure}

\FloatBarrier

The next lemma explains the conditions imposed above. The
degree bound in~\eqref{eq:parameters} makes the initial hypergraph
Berge-$F$-free, while~\eqref{eq:matching} provides the hyperedges needed
for the edges between $R$ and $V(G)$.

\begin{lemma}\label{lem:construction}
Let $\HH_0$ be the hypergraph defined in Construction~\ref{con:initial}.
Then $\HH_0$ is Berge-$F$-free, and every
Berge-$F$-free $k$-uniform hypergraph on the same vertex set containing
$\HH_0$ has $O_{F,k}(n)$ edges. Consequently, an inclusion-maximal
Berge-$F$-free extension of $\HH_0$ is a Berge-$F$-saturated hypergraph
with $O_{F,k}(n)$ edges.
\end{lemma}

\begin{proof}
We first count the vertices of degree at least $d$ in $\HH_0$.
Every vertex of a part $B_i$ has degree $q\le d-1$.
Every vertex in a set $W_{xy}$ has degree one, and every vertex of
$Z_1$ has degree zero. If $w\in V(G)\setminus S$, then $d_F(w)<d$
and $d_{\HH_0}(w)=d_G(w)\le d_F(w)$. Thus only the vertices in $S$
can have degree at least $d$.
In a Berge-$F$, the hypergraph degree of each core vertex is at least
its degree in $F$, since its incident graph edges require distinct
hyperedges containing that vertex. But $|S|<|A|$, because $R$ meets
$A$. There are therefore too few vertices of degree at least $d$ in
$\HH_0$ to serve as the core vertices corresponding to $A$.
Hence $\HH_0$ is Berge-$F$-free.

Let $\HH\supseteq\HH_0$ be any Berge-$F$-free $k$-uniform hypergraph
on the same vertex set, and let $\JJ$ have the same vertex set and
edge set $E(\HH)\setminus E(\HH_0)$. Put $T=F[R]$.
We show that $\JJ$ contains no Berge-$T$ whose core vertices lie
outside $Z$ in pairwise distinct parts $B_i$. All vertices of $T$,
including any isolated vertices, are retained as core vertices.

Suppose such a Berge-$T$ exists. For each $xy\in E(T)$, let
$\theta(xy)$ be the corresponding hyperedge under its defining
bijection. For $v\in R$, let $B(v)$ be the part containing the core
vertex corresponding to $v$. Keep the vertices of $G$ at their fixed
labels. Assign hyperedges to the edges of $F$ as follows:
\begin{equation}\label{eq:edge-assignment}
 \begin{array}{c|l}
 \text{edge of }F&\text{hyperedge in }\HH\\ \hline
 xy\in E(G)&h_{xy}\\
 xy\in E(T)&\theta(xy)\\
 vz,\ v\in R,\ z\in Q_v&B(v)\cup P_{\rho_v(z)}.
 \end{array}
\end{equation}
For a fixed $v$, the last row uses distinct hyperedges by injectivity
of $\rho_v$. Different vertices of $R$ use different parts $B(v)$,
so their assigned hyperedges are also distinct.
The first row uses no part $B_i$, and the second row consists of
edges outside $\HH_0$. Every assigned hyperedge contains the two
core vertices corresponding to its graph edge. The fixed vertices
of $G$ and the core vertices chosen outside $Z$ are also distinct.
These assignments therefore give a Berge-$F$ in $\HH$, a contradiction.

It remains to bound the number of edges of $\JJ$. Put
\[
 p=|R|,\qquad b=|Z|+ap,\qquad
 \gamma=\binom b{k-1}+p.
\]
We claim that $\JJ$ is $\gamma$-degenerate. Suppose, to the contrary,
that a nonempty induced subhypergraph $\JJ[W]$ has minimum degree
greater than $\gamma$. Then $|W|>b$, since otherwise simplicity gives
$d_{\JJ[W]}(x)\le\binom{|W|-1}{k-1}\le\binom b{k-1}<\gamma$.
Root each component of $T$ and embed the components successively,
placing every parent before its children. We keep the chosen core
vertices outside $Z$ in distinct parts and use a different hyperedge
for each graph edge. These restrictions on parts apply only to the
core vertices, not to the other vertices of the assigned hyperedges.

To start a component, choose a vertex of $W$ outside $Z$ and the
parts already used. There are at most $p-1$ such parts, so fewer
than $b$ vertices are excluded; hence a choice exists. This also
handles isolated vertices of $T$.
Suppose next that a parent has been placed at $x$ and a child remains
to be placed. Let $Y$ be the union of $Z$ and the parts already used.
Then $x\in Y$ and $|Y|\le b$. At most $\binom b{k-1}$ hyperedges
through $x$ lie entirely in $Y$, and fewer than $p$ hyperedges have
already been assigned. Since $d_{\JJ[W]}(x)>\gamma$, there is an
unassigned hyperedge $e\ni x$ with $e\nsubseteq Y$. Choose
$y\in e\setminus Y$ for the child and assign $e$ to the corresponding
graph edge. Since $e\subseteq W$, the vertex $y$ lies in $W$ and in
a new part. Repeating this step constructs the forbidden Berge-$T$.

Thus $\JJ$ is $\gamma$-degenerate. Successively delete a vertex of
degree at most $\gamma$ from the remaining induced subhypergraph,
counting each hyperedge when its first vertex is deleted.
Each step counts at most $\gamma$ hyperedges, so $|E(\JJ)|\le\gamma n$.
By~\eqref{eq:fixed-size}, $\gamma$ is bounded
in terms of $F$ and $k$. Together with~\eqref{eq:initial-size}, this
gives $|E(\HH)|=O_{F,k}(n)$. Finally, add missing edges to $\HH_0$
one at a time whenever they do not create a Berge-$F$. Since only
finitely many $k$-edges are available, this procedure terminates at
an inclusion-maximal Berge-$F$-free extension, which is saturated
by definition and satisfies the same bound.
\end{proof}

\FloatBarrier
\subsection{Proof of Theorem~\ref{thm:main}}

\begin{proof}
The cited results in Section~\ref{sec:prelim} allow us to
assume~\eqref{eq:remaining}. Choose the degree threshold
\begin{equation}\label{eq:d}
 d=\min_{xy\in E(F)}\max\{d_F(x),d_F(y)\},\qquad r=d-1,
\end{equation}
and put
\[
 L=\{v\in V(F):d_F(v)<d\},\qquad A=V(F)\setminus L.
\]
Since every component of $F$ contains a cycle, $F$ has no $K_2$
component, and $d\ge2$. The set $L$ is independent, since an edge
with both endpoints in $L$ would contradict~\eqref{eq:d}.
Also, $|A|\ge2$: an independent set together with at most one other
vertex cannot contain a cycle.

There is a vertex of degree exactly $d$. Define
\[
 \ell=\max\{|N_F(u)\cap L|:d_F(u)=d\}.
\]
We choose $R$ and the sets $P_j$ according as $\ell\le1$ or
$\ell\ge2$, and then apply Construction~\ref{con:initial} and
Lemma~\ref{lem:construction}.

\smallskip
\noindent\textit{Case 1: deleting two adjacent vertices ($\ell\le1$).}
If $\ell=0$, choose an edge $uv$ attaining the minimum in~\eqref{eq:d}.
Both endpoints have degree $d$, and all their other neighbors lie in
$A$. If $\ell=1$, choose a vertex $u$ of degree $d$ with its unique
neighbor $v\in L$. All other neighbors of $u$ lie in $A$, and all
other neighbors of $v$ do as well, since $L$ is independent.
Thus, in either case,
\[
 d_F(u)=d\ge d_F(v),\qquad
 (N_F(u)\cup N_F(v))\setminus\{u,v\}\subseteq A.
\]
Take $R=\{u,v\}$, $S=A\setminus R$ and $G=F-R$.
Then $F[R]$ is an edge and every edge from $R$ to $V(G)$ ends in $S$.
Write
\[
 U=N_F(u)\setminus\{v\},\qquad W=N_F(v)\setminus\{u\}.
\]
We have $U,W\subseteq S$, $|U|=r$ and $|W|\le r$.
We will choose $r$ distinct sets of a common size, either one or two,
whose incidence graph has matchings saturating both $U$ and $W$.

If $W\subseteq U$, take the $r$ singleton sets $\{z\}$ for $z\in U$.
They have matchings saturating both $U$ and $W$.
Otherwise, extend $W$ to an $r$-set $W'\subseteq U\cup W$.
Enumerate $U=\{u_1,\ldots,u_r\}$ and $W'=\{w_1,\ldots,w_r\}$,
and initially take $P_i=\{u_i,w_i\}$. These indexed sets have size at
most two and admit matchings saturating both $U$ and $W'$.
Since $W\nsubseteq U$, we have $|S|\ge |U\cup W|\ge r+1$, and hence
$r\le\binom{|S|}{2}$. Lemma~\ref{lem:distinct}, applied with $t=2$,
gives $r$ distinct pairs. Use these pairs as $P_1,\ldots,P_r$.
Their incidence graph has a matching saturating $U$ and a matching
saturating $W'$, and hence a matching saturating $W$.

In both subcases, $q=r=d-1$ and $t$ is one or two. Since $k\ge3$,
conditions~\eqref{eq:parameters} and~\eqref{eq:matching} hold.
With $a=k-t$ and $G=F-R$, form $\HH_0$ as in
Construction~\ref{con:initial}. Lemma~\ref{lem:construction}
gives the required saturated hypergraph.

\smallskip
\noindent\textit{Case 2: deleting an induced star ($\ell\ge2$).}
Choose $u$ with $d_F(u)=d$ and $|N_F(u)\cap L|=\ell$, and put
\begin{equation}\label{eq:star}
 R=\{u\}\cup(N_F(u)\cap L),\qquad S=A\setminus\{u\},
 \qquad G=F-R.
\end{equation}
Since $L$ is independent, $F[R]$ is an induced star and every edge
from $R$ to $V(G)$ ends in $S$. The set $S$ is nonempty since $|A|\ge2$.
Moreover,
\begin{equation}\label{eq:Qsize}
 |Q_u|=d-\ell\le r-1,\qquad
 |Q_v|=d_F(v)-1\le r-1\quad(v\in R\setminus\{u\}).
\end{equation}
Here $d\ge3$: if $d=2$, the two lower-degree neighbors of $u$ would
both have degree one, making their component a three-vertex path,
contrary to~\eqref{eq:remaining}. Thus $r\ge2$.
It remains to find the sets satisfying~\eqref{eq:parameters}
and~\eqref{eq:matching}. Write $c=|S|$.

\smallskip
\noindent\textit{Case 2.1: $c\le k$.}
If $c=1$, take $q=1$, $t=1$ and $P_1=S$. The matching condition holds
because every $Q_v$ is a subset of $S$.
For $2\le c\le k$, put $q=\min\{r,c\}$, choose distinct
$s_1,\ldots,s_q\in S$, and take
\begin{equation}\label{eq:smallS}
 P_j=S\setminus\{s_j\}\qquad(j\in[q]).
\end{equation}
These sets have common size $t=c-1<k$.
To verify~\eqref{eq:matching}, fix $Q_v$ and apply Hall's theorem.
Each vertex of $Q_v$ belongs to at least $q-1\ge1$ of the sets $P_j$.
Any subset of $Q_v$ of size at least two meets every $P_j$, and
therefore has all $q$ indices as neighbors. Since
$|Q_v|\le\min\{r-1,c\}\le q$, Hall's condition holds.

\smallskip
\noindent\textit{Case 2.2: $c>k$.}
Let $K$ be a minimum vertex cover and write $\beta=|K|=\beta(F)$.
Since $c>k\ge\beta$, there is a vertex $w\in S\setminus K$.
Every neighbor of $w$ belongs to $K$, so
\begin{equation}\label{eq:r-bound}
 d\le d_F(w)\le\beta\le k,\qquad r\le k-1.
\end{equation}
Put $B=K\cap S$ and $X=K\cap R$.
The set $X$ is nonempty because $K$ covers the edges of the star $F[R]$.
The set $B$ is also nonempty. Indeed, otherwise
$K\subseteq L\cup\{u\}$, so every vertex of $K$ has degree at most $d$.
All neighbors of $S$ would lie in $K$, and $S\cap K=\varnothing$.
Hence
\[
 dc\le\sum_{s\in S}d_F(s)
 =e_F(S,K)\le\sum_{x\in K}d_F(x)\le d\beta<dc,
\]
a contradiction.

We first choose $r$ indexed subsets, allowing repetitions and different
sizes. Choose $b_0\in B$ and $x_0\in X$.
For each $x\in X$, assign distinct colors from $[r]$ to the edges
from $x$ to $S\setminus K$. Require color $1$ to be unused at $x$
whenever $x=x_0$ or $b_0\in Q_x$.
These assignments can be made independently by~\eqref{eq:Qsize},
since there are at most $r-1$ edges to color at each $x$.
No restriction is imposed on colors at their endpoints in $S\setminus K$.
For $j\in[r]$, let $E_j\subseteq S\setminus K$ be the set of
endpoints of edges assigned color $j$, and put
\begin{equation}\label{eq:cover-sets}
 P_1=B\cup E_1,\qquad
 P_j=(B\setminus\{b_0\})\cup E_j\quad(2\le j\le r).
\end{equation}
Each $x\in X$ contributes at most one vertex to each $E_j$, and $x_0$
contributes none to $E_1$. Consequently,
\begin{equation}\label{eq:cover-size}
 |P_j|\le |B|+|X|-1\le\beta-1\le k-1\qquad(j\in[r]).
\end{equation}

For $x\in X$, match each vertex of $Q_x\setminus K$ to the color
of its edge from $x$. If $b_0\in Q_x$, color $1$ is unused at $x$,
so match $b_0$ to index $1$. Every other vertex of $Q_x\cap B$
belongs to all $P_j$ and can be matched to an unused index.
There are enough indices because $|Q_x|\le r-1$.
For $v\in R\setminus K$, we have $Q_v\subseteq B$.
Match $b_0$ to index $1$ when needed, and match the other vertices
to distinct remaining indices. Thus every $Q_v$ has the required
matching in the indexed family~\eqref{eq:cover-sets}.

Finally, take $t=k-1$. By~\eqref{eq:r-bound} and $c>k$,
\[
 r\le k-1<c\le\binom c{k-1}.
\]
Lemma~\ref{lem:distinct} gives $r$ distinct $(k-1)$-subsets of $S$
preserving every required matching. Use these as $P_1,\ldots,P_r$
and put $q=r$. This proves~\eqref{eq:parameters} and~\eqref{eq:matching}.

In both subcases, $R$ meets $A$, $F[R]$ is a forest, and every edge
from $R$ to $V(G)$ ends in $S$. With the choices of $q$ and $t$
above and $a=k-t$, form $\HH_0$ as in Construction~\ref{con:initial}.
Lemma~\ref{lem:construction} gives the required saturated hypergraph.
This completes the proof of Theorem~\ref{thm:main}.
\end{proof}



\section*{Declaration on the use of generative AI}
During the preparation of this manuscript, the author used GPT-6 Astra
(OpenAI) to explore proof strategies, assist with checking proofs, and
refine the exposition. The author reviewed and verified all mathematical
arguments, results, and conclusions presented in the manuscript and
takes full responsibility for its content.

\end{document}